\documentclass[11pt,oneside,reqno]{amsart}

\usepackage{amssymb}
\usepackage{algpseudocode}
\usepackage{algorithm}
\usepackage{subcaption}
\usepackage[singlelinecheck=false]{caption}
\usepackage{verbatim}
\usepackage{graphicx}
\usepackage[noadjust]{cite}
\RequirePackage{dsfont} 
 \scrollmode
\allowdisplaybreaks
\usepackage{graphicx}
\usepackage{epstopdf}
\usepackage[usenames]{color}
\definecolor{darkred}{RGB}{139,0,0}
\definecolor{darkgreen}{RGB}{0,100,0}
\definecolor{darkmagenta}{RGB}{139,0,139}

\usepackage{amsmath}
\usepackage{tikz}

\makeatletter
\newcommand{\xleftrightarrow}[2][]{\ext@arrow 3359\leftrightarrowfill@{#1}{#2}}
\makeatother

\newcommand{\xdasharrow}[2][->]{
\tikz[baseline=-\the\dimexpr\fontdimen22\textfont2\relax]{
\node[anchor=south,font=\scriptsize, inner ysep=1.5pt,outer xsep=2.2pt](x){#2};
\draw[shorten <=3.4pt,shorten >=3.4pt,dashed,#1](x.south west)--(x.south east);
}
}

\newcommand{\DEBUG}{}

\ifdefined\DEBUG%

  \def\rem#1{{\marginpar{\raggedright\scriptsize #1}}}
  \newcommand{\pmr}[1]{\rem{\color{blue}{$\bullet$ #1}}}
  \newcommand{\ppr}[1]{\rem{\color{red}{$\bullet$ #1}}}
 \else%

  \newcommand{\ppr}[1]{}
  \newcommand{\pmr}[1]{}
 \fi

\theoremstyle{plain}
\newtheorem{theorem}{Theorem}

\newtheorem{corollary}{Corollary}

\theoremstyle{definition}
\newtheorem{remark}{Remark}

\usepackage{enumitem}
\begin{document}

\title
[Exceptional set]
{On the properties of the exceptional set  \\ for the randomized Euler and Runge-Kutta schemes}

\author[T. Bochacik]{Tomasz Bochacik}
\address{AGH University of Science and Technology,
Faculty of Applied Mathematics,\newline
Al. A.~Mickiewicza 30, 30-059 Krak\'ow, Poland}
\email{bochacik@agh.edu.pl}

\begin{abstract}
We show that the probability of the exceptional set decays exponentially for a~broad class of randomized algorithms approximating solutions of ODEs, admitting a certain error decomposition. This class includes randomized explicit and implicit Euler schemes, and the randomized two-stage Runge-Kutta scheme (under inexact information). We design a confidence interval for the exact solution of an IVP and perform numerical experiments to illustrate the theoretical results.
\end{abstract}
\maketitle
\tableofcontents

In this paper we consider randomized versions of the following algorithms approximating solutions of ordinary differential equations (ODEs): explicit Euler scheme, implicit Euler scheme, and two-stage Runge-Kutta scheme. Error bounds for these and other randomized algorithms have been broadly studied in the literature, see \cite{daun1, backward_euler, HeinMilla, JenNeuen, Kac1, KruseWu_1, NOV, stengle2}. In \cite{randRK, randEuler}, the $L^p(\Omega)$-norm of worst-case errors of the three aforementioned schemes has been analysed in the setting of inexact information.

The main concept investigated in this paper is the exceptional set, i.e. the set where the random worst-case error of a given randomized algorithm does not achieve the rate of convergence given by the mean-square error bound. As a well-known example we may recall the Monte Carlo integration. If the integrand is Borel-measurable and bounded, Hoeffding's inequality can be employed to show that the probability of the exceptional set of the crude MC method has an exponential decay, see \cite{pages, mc-pp}. 

Similar approach, based on Azuma's inequality (see \cite{azuma}), has been applied by S. Heinrich and B. Milla to the family of randomized Taylor schemes for ODEs. For these algorithms, the probability of the exponential set also proved to decay exponentially, see Proposition 2 in \cite{HeinMilla}. In this paper, we aim to extend this result in two directions. Firstly, we will cover the algorithms investigated in \cite{randEuler, randRK}, under very mild assumptions considered in these papers. Secondly, our analysis will be performed in the setting of inexact information. 

The structure of this paper is as follows. In section 1 we introduce notation, the class of initial value problems, and the model of computation. We also recall definitions of randomized Euler and Runge-Kutta schemes under inexact information. In section 2 we provide an upper bound for the probability of the exceptional set for each algorithm admitting a certain error decomposition. The general setting considered in this paper covers all algorithms analysed in \cite{randRK, randEuler, HeinMilla}. In section 3 we construct a confidence region for the exact solution of the IVP, based on the inequality established in the previous section. We also carry out numerical experiments which illustrate theoretical findings. Summary of this paper is provided in section 4, as well as directions for further research.

\section{Preliminaries} \label{sec:prel}

\subsection{Notation and the class of IVPs}

Let $\|\cdot\|$ be the one norm in $\mathbb{R}^d$, i.e. $\|x\|=\sum\limits_{k=1}^d|x_k|$ for $x\in \mathbb{R}^d$. For $x\in\mathbb{R}^d$ and $r\in [0,\infty)$, we denote by $B(x,r)=\{y\in\mathbb{R}^d \colon \|y-x\|\leq r\}$ the closed ball in $\mathbb{R}^d$ with center $x$ and radius $r$. Moreover, $B(x,\infty) = \mathbb{R}^d$ for all $x\in\mathbb{R}^d$.

We consider IVPs of the following form: 
\begin{equation}
	\label{eq:ode}
		\left\{ \begin{array}{ll}
			z'(t)= f(t,z(t)), \ t\in [a,b], \\
			z(a) = \eta, 
		\end{array}\right.
\end{equation}
where $-\infty < a < b < \infty$, $\eta\in\mathbb{R}^d$, $f \colon [a,b]\times\mathbb{R}^d\to\mathbb{R}^d$, $d\in\mathbb{Z}_+$. 

As in \cite{randRK, randEuler}, by $F^\varrho_R=F^\varrho_R(a,b,d,\varrho,K,L)$ we denote the class of pairs $\left(\eta,f\right)$ satisfying the following conditions:
\begin{itemize}
\setlength\itemsep{4pt}
\item[(A0)] $\left\|\eta\right\|\leq K$,
\item[(A1)] $f\in \mathcal{C}\left([a,b]\times\mathbb{R}^d\right)$,
\item[(A2)] $\| f(t,x)\| \leq K\left(1+\left\|x\right\|\right)$ for all $(t,x)\in [a,b]\times\mathbb{R}^d$,
\item[(A3)] $\| f(t,x) - f(s,x) \|\leq L|t-s|^\varrho$ for all $t,s\in [a,b], x\in B\left(\eta,R\right)$,
\item[(A4)] $\| f(t,x) - f(t,y) \|\leq L\|x-y\|$ for all $t\in [a,b], x,y \in B\left(\eta,R\right)$.
\end{itemize}
Under the above assumptions, the solution of \eqref{eq:ode} exists and is unique, cf. \cite{randRK}.

\subsection{Model of computation -- general description}

Let $(\Omega,\Sigma,\mathbb{P})$ be a complete probability space and let $\mathcal{N}=\{A\in\Sigma \colon \mathbb{P}(A)=0\}$. By $F$ we denote a certain class of IVPs of the form \eqref{eq:ode} -- or equivalently pairs $(\eta,f)$. For example, we can choose $F=F^\varrho_R$ with specified parameters $a,b,d,\varrho,K,L,R$.

We investigate randomized algorithms approximating solutions of \eqref{eq:ode} based on inexact information about $f$. To this end, we consider the following model of computation, cf. \cite{randRK,randEuler}. We introduce a parameter $\delta$, which will be called a precision parameter (or a noise parameter). By $\mathcal{K}(\delta)$ we denote a~class of noise functions $$\tilde{\delta} \colon [a,b]\times\mathbb{R}^d\to\mathbb{R}^d$$ such that $\tilde{\delta}$ is Borel measurable and may satisfy some additional assumptions (where $\delta$ can be a parameter, e.g. a Lipschitz constant). Let 
\begin{equation} \label{eq:V-f-delta}
    V_{f}(\delta)=\{ \tilde f \colon \exists_{\tilde\delta_f\in\mathcal{K}(\delta)} \ \tilde f =f+\tilde\delta_f\}
\end{equation} 
and 
\begin{equation} \label{eq:V-eta-f}
    V_{(\eta,f)}(\delta) = B(\eta,\delta)\times V_{f}(\delta)
\end{equation}
for $(\eta,f)\in F$ and $\delta\in [0,1]$.

Let $(\eta,f)\in F$ and $(\tilde\eta,\tilde f)\in V_{(\eta,f)}(\delta)$. Any vector of the following form:
\begin{equation*}
    N(\tilde \eta,\tilde f)=[\tilde\eta,\tilde f(t_1,y_1),\ldots,\tilde f(t_{i},y_{i}),\tilde f(\theta_1,z_1),\ldots,\tilde f(\theta_{i},z_{i})],
\end{equation*}
where $i\in\mathbb{N}$, $t_1,\ldots,t_{i}\in [a,b]$ are deterministic values and $(\theta_1,\ldots,\theta_{i})$ is a random vector on  $(\Omega,\Sigma,\mathbb{P})$, will be called a vector of noisy information about $(\eta,f)$ based on $2i$ noisy evaluations of $f$. Moreover, we assume that
\begin{equation*}
    (y_1,z_1)=\psi_1(\tilde\eta)
\end{equation*}
and
\begin{equation*}
    (y_j,z_j)=\psi_j\Bigl(\tilde f(t_1,y_1),\ldots,\tilde f(t_{j-1},y_{j-1}),\tilde f(\theta_1,z_1),\ldots,\tilde f(\theta_{j-1},z_{j-1}),\tilde\eta\Bigr)
\end{equation*}
for Borel measurable mappings $\psi_j:\mathbb{R}^{(2j+1)d}\to\mathbb{R}^d\times \mathbb{R}^d$, $j\in\left\{2,\ldots,i\right\}$. We define the following filtration: $\mathcal{F}_0=\sigma(\mathcal{N})$ and $\mathcal{F}_j =\sigma\bigl( \sigma\left(\theta_1,\ldots,\theta_j\right)\cup\mathcal{N}\bigr)$ for $j\in\{1,\ldots,i\}$.

For a given $n\in\mathbb{N}$, we consider the class $\Phi_n$ of algorithms $\mathcal{A}$ which aim to compute the approximate solution $z$ of \eqref{eq:ode}, using $N(\tilde\eta,\tilde f)$ based on at most $n$ noisy evaluations of $f$. That is, the class $\Phi_n$ contain algorithms of the following form:
\begin{equation*}
\label{def_alg}
    \mathcal{A}(\tilde\eta,\tilde f,\delta)=\varphi(N(\tilde \eta,\tilde f)),
\end{equation*}
where
\begin{equation*}
    \varphi:\mathbb{R}^{(2i+1)d}\to D([a,b];\mathbb{R}^d)
\end{equation*}
is a Borel measurable function. In the Skorokhod space $D([a,b];\mathbb{R}^d)$, endowed with the Skorokhod topology, we consider the Borel $\sigma$-field $\mathcal{B}(D([a,b];\mathbb{R}^d))$.

\subsection{The algorithms} We recall the algorithms which will be investigated in this paper. For each of them, we specify the class of IVPs and the class of noise functions, for which error bounds have been established in \cite{randRK, randEuler}. Notation introduced in this subsection is generally consistent with those articles, however some changes have been introduced in order to facilitate reading of this paper.

Let $n\in\mathbb{Z}_+$, $h=\frac{b-a}{n}$, $t_j = a+jh$ for $j\in\{0,1,\ldots,n\}$, $\theta_j = t_{j-1} + \tau_jh$, $\tau_j\sim U(0,1)$ for $j\in\{1,\ldots,n\}$. We assume that the family of random variables $\{\tau_1,\ldots,\tau_n\}$ is independent. \medskip

\textit{The randomized explicit Euler method under inexact information} is defined as follows:
\begin{equation}\label{eq:explicit_euler}
\bar W^0 = \tilde{\eta},  \ \ \bar W^j = \bar W^{j-1} + h\cdot \tilde f\left(\theta_j, \bar W^{j-1}\right), \ j\in\{1,\ldots,n\}.
\end{equation}
The approximate solution of \eqref{eq:ode} on $[a,b]$ is the function $\bar l^{EE} \colon [a,b]\to\mathbb{R}^d$ given by $\bar l^{EE}(t)=\bar l^{EE}_j(t)$ for $t\in [t_{j-1},t_j]$, where
\begin{equation} \label{eq:le}
\bar l^{EE}_j(t) = \frac{\bar W^{j}-\bar W^{j-1}}{h}(t-t_{j-1})+\bar W^{j-1}, \ \ j\in\{1,\ldots,n\}.
\end{equation}

The randomized explicit Euler scheme has been analyzed in \cite{randEuler} assuming that $(\eta,f)\in F^\varrho_{R_{EE}}$ and $\bigl(\tilde{\eta},\tilde{f}\bigr)\in V^{EE}_{(\eta,f)}(\delta)$, where 
\begin{equation} \label{eq:R-EE}
    R_{EE} = \max\bigl\{ (K+2)e^{(K+1)(b-a)} + K-1 , K(1+b-a)e^{K(b-a)} + K \bigr\},
\end{equation}
the class of noise functions is defined as
\begin{align} \label{eq:K-EE}
\mathcal{K}_{EE}(\delta) & = \left\{ \tilde{\delta} \colon [a,b]\times\mathbb{R}^d\to\mathbb{R}^d \ \colon \ \tilde{\delta} \text{ is Borel measurable}, \right. \notag \\ & \left. \ \ \ \ \ \ \ \ \  \|\tilde{\delta}(t,y)\|\leq\delta\left(1+\left\| y\right\|\right)  \hbox{for all} \ t\in [a,b], y\in \mathbb{R}^d \right\},
\end{align}
and similarly as in \eqref{eq:V-f-delta}--\eqref{eq:V-eta-f}, we define
\begin{equation*}
    V^{EE}_{f}(\delta)=\{ \tilde f \colon \exists_{\tilde\delta_f\in\mathcal{K}_{EE}(\delta)} \ \tilde f =f+\tilde\delta_f\}
\end{equation*} 
and 
\begin{equation*}
    V^{EE}_{(\eta,f)}(\delta) = B(\eta,\delta)\times V^{EE}_{f}(\delta)
\end{equation*}
for $(\eta,f)\in F^\varrho_{R_{EE}}$ and $\delta\in [0,1]$. \medskip

\textit{The randomized implicit Euler scheme under inexact information} is given by the following relation:
\begin{equation}
\label{eq:implicit_euler}
\bar U^0 = \tilde \eta,  \ \ \bar U^j = \bar U^{j-1} + h\cdot \tilde f\bigl(\theta_j, \bar U^{j}\bigr), \ j\in\{1,\ldots,n\}.
\end{equation}
Function $\bar l^{IE} \colon [a,b]\to\mathbb{R}^d$ is defined in the same fashion as $\bar l^{EE}$ (i.e. through the linear interpolation) but this time we use knots $(t_j,\bar U_j)$ instead of $(t_j,\bar W_j)$ for $j\in\{0,1,\ldots,n\}$.

The error bound for this scheme in \cite{randEuler} has been obtained assuming that $(\eta,f)\in F^\varrho_{\infty}$ and $\bigl(\tilde{\eta},\tilde{f}\bigr)\in V^{IE}_{(\eta,f)}(\delta)$, where
\begin{equation} \label{eq:K-IE}
\mathcal{K}_{IE}(\delta) = \left\{ \tilde{\delta} \in \mathcal{K}_{EE}(\delta) \colon \bigl\| \tilde{\delta}(t,x) - \tilde{\delta}(t,y) \bigr\| \leq \delta \|x-y\| \text{ for all }t\in[a,b], x,y\in\mathbb{R}^d \right\}
\end{equation}
and definitions of $V_f^{IE}(\delta)$ and $V^{IE}_{(\eta,f)}(\delta)$ are analogous as definitions formulated above for the explicit scheme. \medskip

\textit{The randomized two-stage Runge Kutta scheme under inexact information} is given by
\begin{equation}
\label{def_rand_RK_1}
		\left\{ \begin{array}{ll}
			\bar V^0 := \tilde\eta,\\
			\bar V_\tau^j := \bar V^{j-1} + h \tau_j \tilde f (t_{j-1},  \bar V^{j-1}), \\
			\bar V^j := \bar V^{j-1} + h \tilde f(\theta_j, \bar V_\tau^j).
		\end{array}\right.,
\end{equation}
where $j\in\{1,\ldots,n\}$. In \cite{randRK} it was assumed that $(\eta,f)\in F^\varrho_{R_{RK}}$ and $\bigl(\tilde{\eta},\tilde{f}\bigr)\in V^{RK}_{(\eta,f)}(\delta)$, where
\begin{eqnarray}
\label{eq:R-RK}
    && R^{RK} = \max\Bigl\{K(1+b-a)\Bigl(1+e^{K(b-a)}(1+K(b-a))\Bigr),\notag\\
    && K+(b-a)(1+K)+\Bigl(\frac{1}{K}+1\Bigr)(1+K(b-a))\Bigl(e^{K(b-a)(1+K(b-a))}(1+K)-1\Bigr)\Bigr\}.
\end{eqnarray}
and
\begin{equation} \label{eq:K-RK}
\mathcal{K}^{RK}(\delta)=\{\tilde \delta:[a,b]\times\mathbb{R}^d\to \mathbb{R}^d \colon \tilde\delta-\hbox{Borel measurable}, \
   \|\tilde\delta(t,y)\|\leq\delta \ \hbox{for all} \ t\in [a,b], y\in \mathbb{R}^d \}.
\end{equation}
Definitions of $\bar l^{RK}$, $V_f^{RK}(\delta)$ and $V^{RK}_{(\eta,f)}(\delta)$ are analogous as for Euler schemes.

\section{Main result}
The following theorem is the generalization of Proposition 2 in \cite{HeinMilla}. As we will see, it can be applied to any randomized algorithm admitting a certain error decomposition.

\begin{theorem} \label{general_theorem}
Let $(\mathcal{A}_n)_{n=1}^\infty$ be a sequence of algorithms approximating the solution of the IVP \eqref{eq:ode} such that $\mathcal{A}_n \in \Phi_n$ and for each $(\eta,f)$ in certain class $F$ and $\bigl(\tilde \eta, \tilde f\bigr)\in V_{(\eta,f)}(\delta)$, the following bound for an approximation error holds with probability $1$ for each $n\in\mathbb{N}$, $n\geq n_0$:
\begin{equation} \label{eq:assump}
\sup_{a\leq t\leq b} \|z(t)-\mathcal{A}_n(t)\| \leq C_1\max_{1\leq k\leq n}\Bigl\| \sum_{j=1}^k E_j(h) \Bigr\|+ C_2h^{\gamma} + C_3\delta, 
\end{equation}
where 
\begin{itemize}
\setlength{\itemsep}{0.5em}
\item $\bigl(E_j(h)\bigr)_{j=1}^n$ is an $\bigl(\mathcal{F}_j\bigr)_{j=1}^n$-adapted process;
\item $\displaystyle \mathbb{E} \bigl(  E_j(h) | \mathcal{F}_{j-1} \bigr) = 0$ and $\|E_j(h)\|\leq C_0h^{\gamma+1/2}$ with probability $1$ for all $j\in\{1,\ldots,n\}$;
\item $C_0,C_1, C_2, C_3>0$ are constants which do not depend on $n$.
\end{itemize}

Then there exist constants $c_1, c_2 >0$ not dependent on $n$, such that for all $(\eta,f)\in F$, $\bigl(\tilde \eta, \tilde f\bigr)\in V_{(\eta,f)}(\delta)$, $n\in\mathbb{N}$, $n\geq n_0$, and for all $\xi \geq c_1$, the algorithm $\mathcal{A}_n$ satisfies
\begin{equation} \label{eq:theorem}
\mathbb{P} \Bigl( \sup_{a\leq t\leq b}\bigl\|z(t)-\mathcal{A}_n(t)\bigr\| >\xi \max\bigl\{h^\gamma, \delta\bigr\} \Bigr) \leq \exp (-c_2\xi^2).
\end{equation}
\end{theorem}

\begin{proof}

Let us define $$E(h) = \max_{1\leq k\leq n} \Bigl\|\sum_{j=1}^k E_j(h)\Bigr\|,$$ $$K_j(h) = \underset{\omega\in\Omega}{\text{ess sup}} \|E_j(h)\|\leq C_0h^{\gamma+1/2}$$ and $$v^2(h) = \sum_{j=1}^n \bigl(K_j(h)\bigr)^2.$$ It is easy to see that 
\begin{equation} \label{eq:v2}
v^2(h) \leq n \cdot C_0^2 h^{2\gamma+1} = C_0^2(b-a)h^{2\gamma}.
\end{equation}

Let us consider any $i\in\{1,\ldots,d\}$ and let $E^i_j(h)$ denotes the $i$-th coodinate of $E_j(h)$.
By Remark 1 in \cite{azuma}, the sequence $(E^i_j(h)/K_j(h))_{j=1}^n$ satisfies the assumptions of Lemma 2 in \cite{azuma}. Hence, for any sequence $(b_j)_{j=1}^n$ of real numbers and for any real number $t$, the following inequality holds:
\begin{equation} \label{eq:azuma}
\mathbb{E} \Bigl[ \exp \Bigl(  t\cdot \max_{1\leq k\leq n} \Bigl| \sum_{j=1}^k b_j \frac{E^i_j(h)}{K_j(h)} \Bigr| \Bigr)\Bigr] \leq 8\exp \Bigl(  \frac{t^2}{2} \sum_{j=1}^n b_j^2 \Bigr). 
\end{equation}
Let us consider an arbitrary $\beta>0$. Then by exponential Chebyshev's inequality (cf. \cite{JakSztencel}, p. 96) and \eqref{eq:azuma} with $b_j = K_j(h)$ for $j\in\{1,\ldots,n\}$ and $t=\frac{\beta h^\gamma}{v^2(h)}$, we obtain
\begin{align*}
\mathbb{P}\Bigl(\max_{1\leq k\leq n} \Bigl| \sum_{j=1}^k E^i_j(h) \Bigr|>\beta h^\gamma\Bigr) & \leq \frac{\mathbb{E}\Bigl[\exp\Bigl(t\cdot\max_{1\leq k\leq n} \Bigl| \sum_{j=1}^k E^i_j(h) \Bigr|\Bigr)\Bigr]}{\exp(\beta h^\gamma t)}  \leq  8\exp\Bigl(- \frac{\beta^2 h^{2\gamma}}{2v^2(h)} \Bigr).
\end{align*}
By \eqref{eq:v2} we get
\begin{equation} \label{eq:Ei}
\mathbb{P}\Bigl(\max_{1\leq k\leq n} \Bigl| \sum_{j=1}^k E^i_j(h) \Bigr|>\beta h^{\gamma}\Bigr) \leq 8\exp\Bigl(- \frac{\beta^2}{2C_0^2(b-a)} \Bigr)
\end{equation}
for any $\beta>0$ and for any $i\in\{1,\ldots,d\}$. Let us note that
\begin{align*}
E(h)=\max_{1\leq k\leq n} \sum_{i=1}^d \Bigl|\sum_{j=1}^k E^i_j(h)\Bigr| \leq d\cdot \max_{1\leq i\leq d} \max_{1\leq k\leq n} \Bigl|\sum_{j=1}^k E^i_j(h)\Bigr|
\end{align*}
with probability $1$. Hence, by \eqref{eq:Ei},
\begin{align} 
\mathbb{P}\bigl(E(h)>\beta h^{\gamma}\bigr) & \leq \mathbb{P}\Bigl(\max_{1\leq i\leq d} \max_{1\leq k\leq n} \Bigl|\sum_{j=1}^k E^i_j(h)\Bigr|>\frac{\beta}{d} h^{\gamma}\Bigr) =  \mathbb{P}\Bigl(\bigcup_{i=1}^d \Bigl\{ \max_{1\leq k\leq n} \Bigl|\sum_{j=1}^k E^i_j(h)\Bigr|>\frac{\beta}{d} h^{\gamma}\Bigr\}\Bigr) \notag \\ & \leq \sum_{i=1}^d \mathbb{P}\Bigl( \max_{1\leq k\leq n} \Bigl|\sum_{j=1}^k E^i_j(h)\Bigr|>\frac{\beta}{d} h^{\gamma}  \Bigr) \leq  8d \exp\Bigl(- \frac{\beta^2}{2C_0^2d^2(b-a)} \Bigr).\label{eq:E}
\end{align}
\medskip

By \eqref{eq:assump} and \eqref{eq:E}, the following inequality holds for $\xi>C_2+C_3$:
\begin{align*}
\mathbb{P} \Bigl( \sup_{a\leq t\leq b} \|z(t)-\mathcal{A}_n(t)\| > \xi \max\bigl\{h^{\gamma}, \delta\bigr\} \Bigr) & \leq \mathbb{P} \Bigl( E(h) > \frac{\xi-C_2-C_3}{C_1} \max\bigl\{h^{\gamma}, \delta\bigr\} \Bigr)\\ & \leq \mathbb{P} \Bigl( E(h) > \frac{\xi-C_2-C_3}{C_1} h^{\gamma} \Bigr)\\  & \leq  8d\exp\left(- \frac{\xi^2}{2C_0^2C_1^2d^2(b-a)}\cdot\Bigl( \frac{\xi-C_2-C_3}{\xi} \Bigr)^2 \right).
\end{align*}

For sufficiently large $\xi$ we have
$$ \Bigl( \frac{\xi-C_2-C_3}{\xi} \Bigr)^2 \geq \frac12 \ \text{ and } \  8d \leq  \exp\left( \frac{1}{8C_0^2C_1^2d^2(b-a)} \xi^2 \right).$$
As a result,
\begin{equation*}
\mathbb{P} \Bigl( \sup_{a\leq t\leq b} \|z(t)-\mathcal{A}_n(t)\| > \xi \max\bigl\{h^{\gamma}, \delta\bigr\} \Bigr) \leq \exp\left(- \frac{1}{8C_0^2C_1^2d^2(b-a)} \xi^2 \right)
\end{equation*}
for sufficiently big $\xi$, which completes the proof.
\end{proof}

\begin{remark} \label{remark1}
If the assumptions of Theorem \ref{general_theorem} are satisfied with $\gamma>\frac12$, then $\mathcal{A}_n(t) \to z(t)$ with probability $1$, uniformly on $[a,b]$, as $n\to\infty$ and $\delta\to 0$.
In fact, 
$$\max_{1\leq k\leq n} \Bigl\|\sum_{j=1}^k E_j(h)\Bigr\| \leq \sum_{j=1}^n \| E_j(h)\| \leq C_0(b-a)h^{\gamma-1/2}$$ and by \eqref{eq:assump} we obtain 
$$\sup_{a\leq t\leq b}\bigl\|z(t)-\mathcal{A}_n(t)\bigr\| \leq C_0C_1(b-a)h^{\gamma-1/2}+C_2h^\gamma+C_3\delta \longrightarrow 0$$ with probability $1$ when $n\to\infty$ and $\delta\to 0$. 

The property pointed out in this remark implies that $\mathcal{A}_n(t)$ is a strongly consistent estimator of $z(t)$ for each $t\in [a,b]$. 
\end{remark}

In the next three corollaries we apply Theorem \ref{general_theorem} to randomized explicit and implicit Euler schemes, and to the randomized two-stage Runge-Kutta scheme.

\begin{corollary} \label{cor:explicit}
There exist constants $c_1, c_2 >0$ dependent only on $a,b,d,K,L$, such that for all $(\eta,f)\in F^\varrho_{R_{EE}}$, $\bigl(\tilde \eta, \tilde f\bigr)\in V^{EE}_{(\eta,f)}(\delta)$, $n\in\mathbb{N}$, $n\geq \lfloor b-a \rfloor +1$, and for all $\xi \geq c_1$, the randomized explicit Euler scheme satisfies
$$\mathbb{P} \Bigl( \sup_{a\leq t\leq b}\bigl\|z(t)-\bar l^{EE}(t)\bigr\| >\xi \max\bigl\{h^{\min\{\varrho+1/2, 1\}}, \delta\bigr\} \Bigr) \leq  \exp (-c_2\xi^2).$$
\end{corollary}

\begin{proof}
Let $$E_j(h) = \int\limits_{t_{j-1}}^{t_j} z'(s)\,\mathrm{d}s - hz'(\theta_j) \ \ \text{for} \ \ j\in\{1,\ldots,n\}.$$

Let us note that $\mathbb{E} \bigl(  E_j(h) | \mathcal{F}_{j-1} \bigr) = \mathbb{E}\bigl(E_j(h)\bigr) = 0$ and $\|E_j(h)\| \leq C_0h^{\min\{\varrho+1/2, 1\}+1/2}$ with probability $1$ for all $j\in\{1,\ldots,n\}$, where $C_0=C_0(a,b,K,L)>0$. To prove the last inequality, we use Lemma 2(ii) from \cite{randEuler}. Specifically, for each $j\in\{1,\ldots,n\}$ we have
\begin{equation} \label{eq:Ej}
  \|E_j(h)\|  \leq \int\limits_{t_{j-1}}^{t_j} \bigl\|z'(s) - z'(\theta_j)\bigr\|\,\mathrm{d}s \leq C_0\int\limits_{t_{j-1}}^{t_j} \bigl|s - \theta_j |^\varrho\,\mathrm{d}s \leq C_0\int\limits_{t_{j-1}}^{t_j} h^\varrho\,\mathrm{d}s = C_0 h^{\varrho+1}.  
\end{equation}

From (14), (15) and arguments between (15) and (16) in \cite{randEuler} we obtain
\begin{equation} \label{eq:nier1}
\sup_{a\leq t\leq b} \|z(t)-\bar l^{EE}(t)\| \leq \alpha_1 h^{\varrho+1} + \max_{1\leq j\leq n} \|z(t_j)-W^j\|+\max_{0\leq j\leq n} \|\bar W^j-W^j\|
\end{equation}
with probability $1$ for some $\alpha_1=\alpha_1(a,b,d,K,L)>0$. By $(W^j)_{j=0}^n$ we denote approximations produced by the explicit Euler scheme under exact information (i.e. when $\delta=0$).  Moreover,
\begin{equation}\label{eq:nier2}
\max_{1\leq j\leq n} \|z(t_j)-W^j\| \leq e^{L(b-a)} \cdot \Bigl( \max_{1\leq k\leq n} \Bigl\| \sum_{j=1}^k E_j(h) \Bigr\| + \max_{1\leq k\leq n} \bigl\| S_2^k \bigr\| \Bigr)
\end{equation}
with probability $1$, which can be shown in the same fashion as inequality (18) in \cite{randEuler}. Furthermore, by (20) in \cite{randEuler}, we have 
\begin{equation}\label{eq:nier3}
\max_{1\leq k\leq n} \bigl\| S_2^k \bigr\| \leq \alpha_2 h
\end{equation}
with probability $1$ for some $\alpha_2=\alpha_2(a,b,K,L)>0$. Fact 2 in \cite{randEuler} implies that  there exists $\alpha_3=\alpha_3(a,b,K,L)>0$ such that
\begin{equation}\label{eq:nier4}
\max_{0\leq j\leq n} \|\bar W^j-W^j\| \leq \alpha_3\delta
\end{equation}
with probability $1$. By \eqref{eq:nier1}, \eqref{eq:nier2}, \eqref{eq:nier3} and \eqref{eq:nier4} we obtain 
\begin{align*}
\sup_{a\leq t\leq b} \|z(t)-\bar l^{EE}(t)\| & \leq \alpha_1 h^{\varrho+1} +e^{L(b-a)} \cdot \Bigl( \max_{1\leq k\leq n}\Bigl\| \sum_{j=1}^n E_j(h) \Bigr\| + \alpha_2 h\Bigr) + \alpha_3\delta \\
& \leq e^{L(b-a)} \max_{1\leq k\leq n}\Bigl\| \sum_{j=1}^n E_j(h) \Bigr\| + \bigl(\alpha_1 + \alpha_2e^{L(b-a)}\bigr)h^{\min\{\varrho+1/2, 1\}} + \alpha_3\delta
\end{align*}
with probability $1$. The desired claim follows from Theorem \ref{general_theorem}.
\end{proof}

\begin{corollary} \label{cor:implicit}
There exist constants $c_1, c_2 >0$ dependent only on $a,b,d,K,L$, such that for all $(\eta,f)\in F^\varrho_{\infty}$, $\bigl(\tilde \eta, \tilde f\bigr)\in V^{IE}_{(\eta,f)}(\delta)$, $n\in\mathbb{N}$, $n\geq \lfloor b-a \rfloor +1$ such that $h(K+1)\leq\frac12$, $hL\leq\frac12$, and for all $\xi \geq c_1$, the randomized implicit Euler scheme satisfies
$$\mathbb{P} \Bigl( \sup_{a\leq t\leq b}\bigl\|z(t)-\bar l^{IE}(t)\bigr\| >\xi \max\bigl\{h^{\min\{\varrho+1/2, 1\}}, \delta\bigr\} \Bigr) \leq  \exp (-c_2\xi^2).$$
\end{corollary}

\begin{proof}
Let $E_j(h)$ for $j\in\{1,\ldots,n\}$ be defined as in the proof of Corollary \ref{cor:explicit}.
Of course $\mathbb{E} \bigl(  E_j(h) | \mathcal{F}_{j-1} \bigr) = \mathbb{E}\bigl(E_j(h)\bigr) = 0$ and $\|E_j(h)\| \leq C_0h^{\min\{\varrho+1/2, 1\}+1/2}$ with probability $1$ for all $j\in\{1,\ldots,n\}$, where $C_0=C_0(a,b,K,L)>0$. From the proof of Theorem 2 in \cite{randEuler} we conclude that 
\begin{align*}
     \sup_{a\leq t \leq b} & \bigl\| z(t)-\bar l^{IE}(t)\bigr\| \leq \alpha_1 h^{\varrho+1} + \max_{0\leq j\leq n}\|z(t_j)-U^j\|+ \max_{0\leq j\leq n}\|U^j-\bar U^j\| \\ & \leq e^{L(b-a)} \cdot \max_{1\leq k\leq n}\Bigl\|\sum_{j=1}^k E_j(h)\Bigr\|+\alpha_1 h^{\min\{\varrho+1/2, 1\}} + \alpha_2\delta
\end{align*}
with probability $1$ for some $\alpha_1=\alpha_1(a,b,d,K,L)>0$ and $\alpha_2 = \alpha_2 (a,b,K,L)>0$. The first passage above follows from the first two lines of the proof of Theorem 2 in \cite{randEuler} -- compare also with \eqref{eq:nier1} in this paper. The second passage can be justified using two lines before (28), (28), inequality after (28) and before (29), and Fact 3 in \cite{randEuler}. It remains to use Theorem \ref{general_theorem} to conclude the proof.
\end{proof}

\begin{corollary} \label{cor:rk}
There exist constants $c_1, c_2 >0$ dependent only on $a,b,d,K,L$, such that for all $(\eta,f)\in F^\varrho_{R_{RK}}$, $\bigl(\tilde \eta, \tilde f\bigr)\in V^{RK}_{(\eta,f)}(\delta)$, $n\in\mathbb{N}$, $n\geq \lfloor b-a \rfloor +1$, and for all $\xi \geq c_1$, the randomized two-stage Runge-Kutta scheme satisfies
$$\mathbb{P} \Bigl( \sup_{a\leq t\leq b}\bigl\|z(t)-\bar l^{RK}(t)\bigr\| >\xi \max\bigl\{h^{\varrho+1/2}, \delta\bigr\} \Bigr) \leq  \exp (-c_2\xi^2).$$
\end{corollary}

\begin{proof}
Let $E_j(h)$ for $j\in\{1,\ldots,n\}$ be defined as in the proof of Corollary \ref{cor:explicit}.
We have $\mathbb{E} \bigl(  E_j(h) | \mathcal{F}_{j-1} \bigr) = \mathbb{E}\bigl(E_j(h)\bigr) = 0$ and $\|E_j(h)\| \leq C_0h^{\varrho+1}$ with probability $1$ for all $j\in\{1,\ldots,n\}$, where $C_0=C_0(a,b,K,L)>0$, cf. \eqref{eq:Ej} in this paper and (22) in \cite{randRK}. We have 
\begin{align*} 
     \sup_{a\leq t \leq b} \bigl\| z(t)-\bar l^{RK}(t)\bigr\| & \leq \alpha_1 h^{\varrho+1} + 3\max_{0\leq j\leq n}\|z(t_j)-V^j\|+ 3\max_{0\leq j\leq n}\|V^j-\bar V^j\| \\ & \leq \alpha_1 h^{\varrho+1/2} + \alpha_2 \Bigl( \max_{1\leq k\leq n}\Bigl\| \sum_{j=1}^n E_j(h) \Bigr\| +  h^{\varrho+1/2}\Bigr) + \alpha_3\delta
\end{align*}
with probability $1$, where constants $\alpha_1,\alpha_2,\alpha_3>0$ depend on $a,b,d,K,L$. For justification see (37), (38), (56) in \cite{randRK} for the first passage, and (45), (47) in \cite{randRK} combined with discrete Gronwall's lemma for the second passage. The desired claim follows from Theorem \ref{general_theorem}.
\end{proof}

\begin{remark}
We note that the randomized two-stage Runge-Kutta scheme considered in \cite{randRK} is the special case of randomized Taylor schemes considered in \cite{HeinMilla}. However, in \cite{randRK} and in Corollary \ref{cor:rk}, we assume only local Lipschitz and H\"older conditions, see Assumptions (A3) and (A4), and we allow noisy evaluations of $f$. Thus, Corollary \ref{cor:rk} is not a special case of Proposition 2 in \cite{HeinMilla}, where global Lipschitz and H\"older conditions as well as exact information were assumed.
\end{remark}

\section{Confidence regions and numerical experiments}

\subsection{Construction of the confidence region}
The following corollary is an immediate consequence of Theorem \ref{general_theorem}. By a suitable choice of $\xi$ in \eqref{eq:theorem} one may construct the confidence region for the exact solution of \eqref{eq:ode}.

\begin{corollary} \label{conf_reg}
Let $c_1,c_2$ be the constants which have appeared in Theorem \ref{general_theorem} and let $$\xi(\varepsilon) = \Bigl( \frac{-\ln\varepsilon}{c_2} \Bigr)^{\frac12}$$ for $\varepsilon \in (0,1)$. Under the assumptions of Theorem \ref{general_theorem} it holds that
\begin{equation*}
\mathbb{P} \Bigl( \sup_{a\leq t\leq b}\bigl\|z(t)-\mathcal{A}_n(t)\bigr\| \leq \xi(\varepsilon) \max\bigl\{h^\gamma, \delta\bigr\} \Bigr) \geq 1-\varepsilon.
\end{equation*}
for all $(\eta,f)\in F$, $\bigl(\tilde \eta, \tilde f\bigr)\in V_{(\eta,f)}(\delta)$, $n\in\mathbb{N}$, $n\geq n_0$ and $0< \varepsilon\leq \exp(-c_1^2c_2)$.
\end{corollary}

\begin{remark} \label{remark3}
With $\max\{h,\delta\}$ approaching $0$, the confidence region for $z$ constructed in Corollary \ref{conf_reg} tightens, whereas the confidence level $1-\varepsilon$ remains unchanged. Hence, this confidence region is uniform with respect to $\max\{h,\delta\}$. 
\end{remark}

\subsection{Numerical experiments}
Let us consider the following test problems:
\begin{equation} \label{eq:A}
\left\{ \begin{array}{ll}
	z'(t)= 2tz(t), \ t\in [0,1], \\
	z(0) = 1
\end{array}\right. \tag{$A$}
\end{equation}
and
\begin{equation} \label{eq:B}
\left\{ \begin{array}{ll}
	z'(t)= \cos\bigl(z^2(t)\bigr), \ t\in [0,1], \\
	z(0) = 1.
\end{array}\right. \tag{$B$}
\end{equation}
The exact solution of \eqref{eq:A} is $z_A(t)=\exp(t^2)$ for $t\in [0,1]$. By $z_B$ we denote the exact solution of \eqref{eq:B}. Let $\bar l^{S,X}_n$ be the approximate solutions of test problem $X\in\{A,B\}$, generated by scheme $S\in\{EE,RK\}$ (the randomized explicit Euler scheme or the randomized two-stage Runge-Kutta scheme, respectively) with $n$ steps. We note that $\varrho=\frac{3}{2}$ for both test problems, cf. assumption (A3). Thus, $\gamma^{RK}=\frac32$ and $\gamma^{EE}=1$, cf. \eqref{eq:theorem}, Corollary \ref{cor:explicit} and Corollary \ref{cor:rk}.

We consider the setting where the class of IVPs is restricted to the test problem ($A$ or $B$), and the class of algorithms is restricted to the randomized explicit Euler scheme or the randomized two-stage Runge-Kutta scheme. By Corollary \ref{conf_reg}, 
\begin{equation} \label{eq:conf_int}
\mathbb{P} \Bigl( \sup_{a\leq t\leq b}\bigl\|z_X(t)-\bar l_n^{S,X}(t)\bigr\| \leq \xi^{S,X}(\varepsilon) \cdot \max\bigl\{n^{-\gamma^S},\delta\bigr\} \Bigr) \geq 1-\varepsilon
\end{equation}
for $0< \varepsilon\leq \exp(-c_1^2c_2)$, where $X\in \{A,B\}$, $S\in\{EE,RK\}$, $\displaystyle \xi^{S,X}(\varepsilon) = \Bigl( \frac{-\ln\varepsilon}{c_2} \Bigr)^{\frac12}$ and $c_1, c_2$ are some positive constants, dependent on $S$ and $X$. 

Since the constants $c_1, c_2$ are not known, we propose the following approach to illustrate the property \eqref{eq:conf_int}. Let $\xi_{\varepsilon,n, \delta}^{S,X}$ be given by the following equation: 
\begin{equation} \label{eq:xi_def}
\mathbb{P} \Bigl( \sup_{a\leq t\leq b}\bigl\|z_X(t)-\bar l_{n}^{S,X}(t)\bigr\| \leq \xi_{\varepsilon,n, \delta}^{S,X} \cdot \max\bigl\{n^{-\gamma^{S}},\delta\bigr\} \Bigr) = 1-\varepsilon.
\end{equation}
In practice, we perform $N$ Monte Carlo simulations of $\displaystyle \sup_{a\leq t\leq b}\bigl\|z_X(t)-\bar l_{n}^{S,X}(t)\bigr\|$, sort the obtained values in the increasing order: $r_{1:N}\leq \ldots\leq r_{N:N}$, and consider the following estimator:

$$\widehat{\xi_{\varepsilon,n,\delta}^{S,X}} = r_{\left\lceil (1-\varepsilon)N \right\rceil:N} \cdot \Bigl( \max\bigl\{n^{-\gamma^{S}},\delta\bigr\}\Bigr)^{-1}.$$

In general, $\xi_{\varepsilon,n,\delta}^{S,X}$ is only a~lower bound for $\xi^{S,X}(\varepsilon)$. However, as we will see, some patterns can be noticed in the behaviour of $\widehat{\xi_{\varepsilon,n,\delta}^{S,X}}$ for different choices of $n$ and $\delta$ (at least for the considered test problems). Thus, with a certain degree of caution, we may provide estimates of $\xi^{S,X}(\varepsilon)$ for $S\in\{EE,RK\}$ and $X\in\{A,B\}$. 

\begin{table}[H]
\centering
\begin{tabular}{  r | c c c c c c c } 
$	n	$ & &  $	\delta=0	$ & $	\delta(n)=n^{-1.1}	$ & $	\delta(n)=n^{-1}	$ & $	\delta(n)=n^{-0.9}	$ & $	\delta=2\cdot 10^{-3}	$ & $	\delta=10^{-4}	$ \\ \hline
$	10	$ & &  $	2.29	$ & $	2.82	$ & $	3.04	$ & $	2.64	$ & $	2.30	$ & $	2.29	$ \\
$	20	$ & &  $	2.23	$ & $	2.58	$ & $	2.79	$ & $	2.29	$ & $	2.24	$ & $	2.23	$ \\
$	50	$ & &  $	2.12	$ & $	2.31	$ & $	2.48	$ & $	1.86	$ & $	2.12	$ & $	2.12	$ \\
$	100	$ & &  $	2.04	$ & $	2.17	$ & $	2.30	$ & $	1.60	$ & $	2.06	$ & $	2.04	$ \\
$	200	$ & &  $	1.98	$ & $	2.06	$ & $	2.17	$ & $	1.39	$ & $	2.02	$ & $	1.98	$ \\
$	500	$ & &  $	1.92	$ & $	1.97	$ & $	2.04	$ & $	1.18	$ & $	2.04	$ & $	1.92	$ \\
$	1\,000	$ & &  $	1.89	$ & $	1.92	$ & $	1.97	$ & $	1.06	$ & $	1.05	$ & $	1.89	$ \\
$	2\,000	$ & &  $	1.87	$ & $	1.89	$ & $	1.93	$ & $	0.95	$ & $	0.56	$ & $	1.87	$ \\
$	5\,000	$ & &  $	1.85	$ & $	1.86	$ & $	1.89	$ & $	0.84	$ & $	0.25	$ & $	1.86	$
\end{tabular}
\caption{$\widehat{\xi_{\varepsilon,n,\delta}^{EE,A}}$ for $\varepsilon=0.05$ and different choices of~$n$ and $\delta$; $N=10^5$.}
\label{tab:EE-A}
\end{table}
\begin{table}[H]
\centering
\begin{tabular}{  r | c c c c c c c } 
$	n	$ & & $	\delta=0	$ & $	\delta(n)=n^{-1.1}	$ & $	\delta(n)=n^{-1}	$ & $	\delta(n)=n^{-0.9}	$ & $	\delta=2\cdot 10^{-3}	$ & $	\delta=10^{-4}	$ \\ \hline
$	10	$ & & $	0.166	$ & $	0.546	$ & $	0.659	$ & $	0.636	$ & $	0.173	$ & $	0.166	$ \\
$	20	$ & & $	0.159	$ & $	0.411	$ & $	0.512	$ & $	0.482	$ & $	0.170	$ & $	0.160	$ \\
$	50	$ & & $	0.157	$ & $	0.299	$ & $	0.376	$ & $	0.336	$ & $	0.175	$ & $	0.157	$ \\
$	100	$ & & $	0.156	$ & $	0.248	$ & $	0.308	$ & $	0.258	$ & $	0.183	$ & $	0.157	$ \\
$	200	$ & & $	0.155	$ & $	0.216	$ & $	0.261	$ & $	0.203	$ & $	0.195	$ & $	0.157	$ \\
$	500	$ & & $	0.155	$ & $	0.189	$ & $	0.221	$ & $	0.152	$ & $	0.221	$ & $	0.158	$ \\
$	1\,000	$ & & $	0.155	$ & $	0.177	$ & $	0.201	$ & $	0.126	$ & $	0.125	$ & $	0.159	$ \\
$	2\,000	$ & & $	0.155	$ & $	0.169	$ & $	0.187	$ & $	0.106	$ & $	0.074	$ & $	0.161	$ \\
$	5\,000	$ & & $	0.155	$ & $	0.163	$ & $	0.175	$ & $	0.087	$ & $	0.039	$ & $	0.164	$
\end{tabular}
\caption{$\widehat{\xi_{\varepsilon,n,\delta}^{EE,B}}$ for $\varepsilon=0.05$ and different choices of~$n$ and $\delta$; $N=10^5$.}
    \label{tab:EE-B}
\end{table}
\begin{table}[H]
\centering
\begin{tabular}{  r | c c c c c c c} 
$	n	$ & & $	\delta=0	$ & $	\delta(n)=n^{-1.6}	$ & $	\delta(n)=n^{-1.5}	$ & $	\delta(n)=n^{-1.4}	$ & $	\delta=2\cdot 10^{-3}	$ & $	\delta=10^{-4}	$ \\ \hline
$	10	$ & & $	4.78	$ & $	4.82	$ & $	4.86	$ & $	3.87	$ & $	4.79	$ & $	4.78	$ \\
$	20	$ & & $	5.17	$ & $	5.17	$ & $	5.22	$ & $	3.87	$ & $	5.20	$ & $	5.18	$ \\
$	50	$ & & $	5.46	$ & $	5.47	$ & $	5.46	$ & $	3.71	$ & $	5.48	$ & $	5.48	$ \\
$	100	$ & & $	5.60	$ & $	5.63	$ & $	5.59	$ & $	3.55	$ & $	2.79	$ & $	5.59	$ \\
$	200	$ & & $	5.70	$ & $	5.67	$ & $	5.70	$ & $	3.36	$ & $	1.02	$ & $	5.71	$ \\
$	500	$ & & $	5.74	$ & $	5.78	$ & $	5.76	$ & $	3.10	$ & $	0.28	$ & $	5.18	$ \\
$	1\,000	$ & & $	5.81	$ & $	5.79	$ & $	5.79	$ & $	2.89	$ & $	0.12	$ & $	1.83	$ \\
$	2\,000	$ & & $	5.83	$ & $	5.83	$ & $	5.82	$ & $	2.73	$ & $	0.06	$ & $	0.65	$ \\
$	5\,000	$ & & $	5.85	$ & $	5.86	$ & $	5.85	$ & $	2.49	$ & $	0.04	$ & $	0.17	$
\end{tabular}
\caption{$\widehat{\xi_{\varepsilon,n,\delta}^{RK,A}}$ for $\varepsilon=0.05$ and different choices of~$n$ and $\delta$; $N=10^5$.}
    \label{tab:RK-A}
\end{table}
\begin{table}[H]
\centering
\begin{tabular}{  r | c c c c c c c } 
$	n	$ &  &   $	\delta=0	$ & $	\delta(n)=n^{-1.6}	$ & $	\delta(n)=n^{-1.5}	$ & $	\delta(n)=n^{-1.4}	$ & $	\delta=2\cdot 10^{-3}	$ & $	\delta=10^{-4}	$ \\ \hline
$	10	$ &  &  $	0.485	$ & $	0.527	$ & $	0.550	$ & $	0.461	$ & $	0.486	$ & $	0.485	$ \\
$	20	$ &  &  $	0.459	$ & $	0.480	$ & $	0.494	$ & $	0.386	$ & $	0.462	$ & $	0.461	$ \\
$	50	$ &  &  $	0.455	$ & $	0.464	$ & $	0.471	$ & $	0.330	$ & $	0.465	$ & $	0.458	$ \\
$	100	$ &  &  $	0.460	$ & $	0.463	$ & $	0.466	$ & $	0.301	$ & $	0.244	$ & $	0.459	$ \\
$	200	$ &  &  $	0.463	$ & $	0.464	$ & $	0.467	$ & $	0.279	$ & $	0.102	$ & $	0.464	$ \\
$	500	$ &  & $	0.467	$ & $	0.466	$ & $	0.469	$ & $	0.254	$ & $	0.045	$ & $	0.419	$ \\
$	1\,000	$ & &  $	0.469	$ & $	0.469	$ & $	0.469	$ & $	0.237	$ & $	0.030	$ & $	0.150	$ \\
$	2\,000	$ & &  $	0.471	$ & $	0.471	$ & $	0.470	$ & $	0.221	$ & $	0.021	$ & $	0.056	$ \\
$	5\,000	$ & &  $	0.471	$ & $	0.470	$ & $	0.472	$ & $	0.202	$ & $	0.013	$ & $	0.018	$
\end{tabular}
\caption{$\widehat{\xi_{\varepsilon,n,\delta}^{RK,B}}$  for $\varepsilon=0.05$ and different choices of~$n$ and $\delta$; $N=10^5$.}
    \label{tab:RK-B}
\end{table}

In the performed tests, each evaluation of $f$ has been disrupted by a random noise (in cases other than $\delta = 0$). Specifically, $\tilde f(t,x)$ has been simulated as $f(t,x) + (1+|x|)\cdot e$ (for $S=EE$) or $f(t,x) + e$ (for $S=RK$), where $e$ is taken from the uniform distribution on $[-\delta,\delta]$, independently for all noisy evaluations of $f$. Thus, in numerical experiments we have further restricted the class of noise functions specified by \eqref{eq:K-EE} and \eqref{eq:K-RK}.

Moreover, we have assumed that the exact solution $z_B$ of \eqref{eq:B} can be replaced by $l^{RK,B}_{10^8}$ (under exact information). We note that some deviation in estimates of $\xi_{\varepsilon,n,\delta}^{S,X}$ may be attributed to errors inherited from MC simulations. 

Based on the results displayed in Tables \ref{tab:EE-A}--\ref{tab:RK-B}, we make the following observations. 
\begin{itemize}[leftmargin=27pt]
\item Typically $\widehat{\xi_{\varepsilon,n,\delta}^{S,X}}$ is close to $\widehat{\xi_{\varepsilon,n,0}^{S,X}}$ if $\delta\leq n^{-\gamma^S}$. However, the case $\delta=2\cdot 10^{-3}$ in Table \ref{tab:EE-B} does not follow this behaviour. Other exceptions are observed in cases $\delta(n) = n^{-\gamma^S}$ and $\delta(n) = n^{-\gamma^S-0.1}$ for small values of $n$.
\item In case $0\leq \delta\leq n^{-\gamma^S}$, estimates $\widehat{\xi_{\varepsilon,n,\delta}^{S,X}}$ appear to stabilise when $n$ increases. This means that the probability of hitting the confidence region \eqref{eq:conf_int} is stable for sufficiently big values of $n$, provided that $\delta$ is bounded by $n^{-\gamma^S}$.
\item Generally $\widehat{\xi_{\varepsilon,n,\delta}^{S,X}}< \widehat{\xi_{\varepsilon,n,0}^{S,X}}$ when $\delta>n^{-\gamma^S}$, which indicates that assumptions imposed on the noise functions, cf. \eqref{eq:K-EE} and \eqref{eq:K-RK}, can be relaxed. Some exceptions are observed in Tables \ref{tab:EE-A} and \ref{tab:EE-B} for small values of $n$ in the case $\delta(n)=n^{-0.9}$. When $\delta>n^{-\gamma^S}$, the expression $\widehat{\xi_{\varepsilon,n,\delta}^{S,X}}$ appears to underestimate $\xi^{S,X}(\varepsilon)$. 
\item When $n$ is fixed, $\widehat{\xi_{\varepsilon,n,\delta}^{S,X}}$ seems to achieve its maximum for $\delta=n^{-\gamma^S}$. This is in line with the intuition as $\delta=n^{-\gamma^S}$ is the maximal value of $\delta$ with no impact on $\max\bigl\{n^{-\gamma^{S}},\delta\bigr\}$, cf. \eqref{eq:xi_def}.
\item Based on the above, we suppose that $\xi^{EE,A} \approx 3$, $\xi^{EE,B} \approx 0.7$, $\xi^{RK,A} \approx 5.9$, and $\xi^{RK,B} \approx 0.6$. However, since it is impossible to test numerically all possible choices of $n$ and $\delta$, these approximations may be valid only under additional conditions on $n$ and $\delta$. 
\end{itemize}

In Figures \ref{fig:conf-reg-A} and \ref{fig:conf-reg-B}, we plotted sample confidence regions for test problems \eqref{eq:A} and \eqref{eq:B}, respectively, based on randomized explicit Euler and two-stage Runge-Kutta schemes. We used $n=25$ steps with noise level $\delta=n^{-\gamma^S}$, $S\in \{EE,RK\}$. We took $\xi^{EE,A} = 3$, $\xi^{EE,B} = 0.7$, $\xi^{RK,A} = 5.9$, and $\xi^{RK,B} = 0.6$ in order to achieve the confidence level of $1-\varepsilon=0.95$ (cf. the last bullet above). Confidence regions are shaded in navy blue. In Figure \ref{fig:conf-reg-A}, the white curve represents the exact solution $z_A$ of \eqref{eq:A}, whereas in Figure \ref{fig:conf-reg-B} -- the approximated solution $l^{RK,B}_{10^8}$ of \eqref{eq:B} obtained through the randomized RK scheme with the large number of steps.

\begin{figure}[H]
    \centering
    \begin{subfigure}{0.45\textwidth}
    \centering
        \includegraphics[width=.95\linewidth]{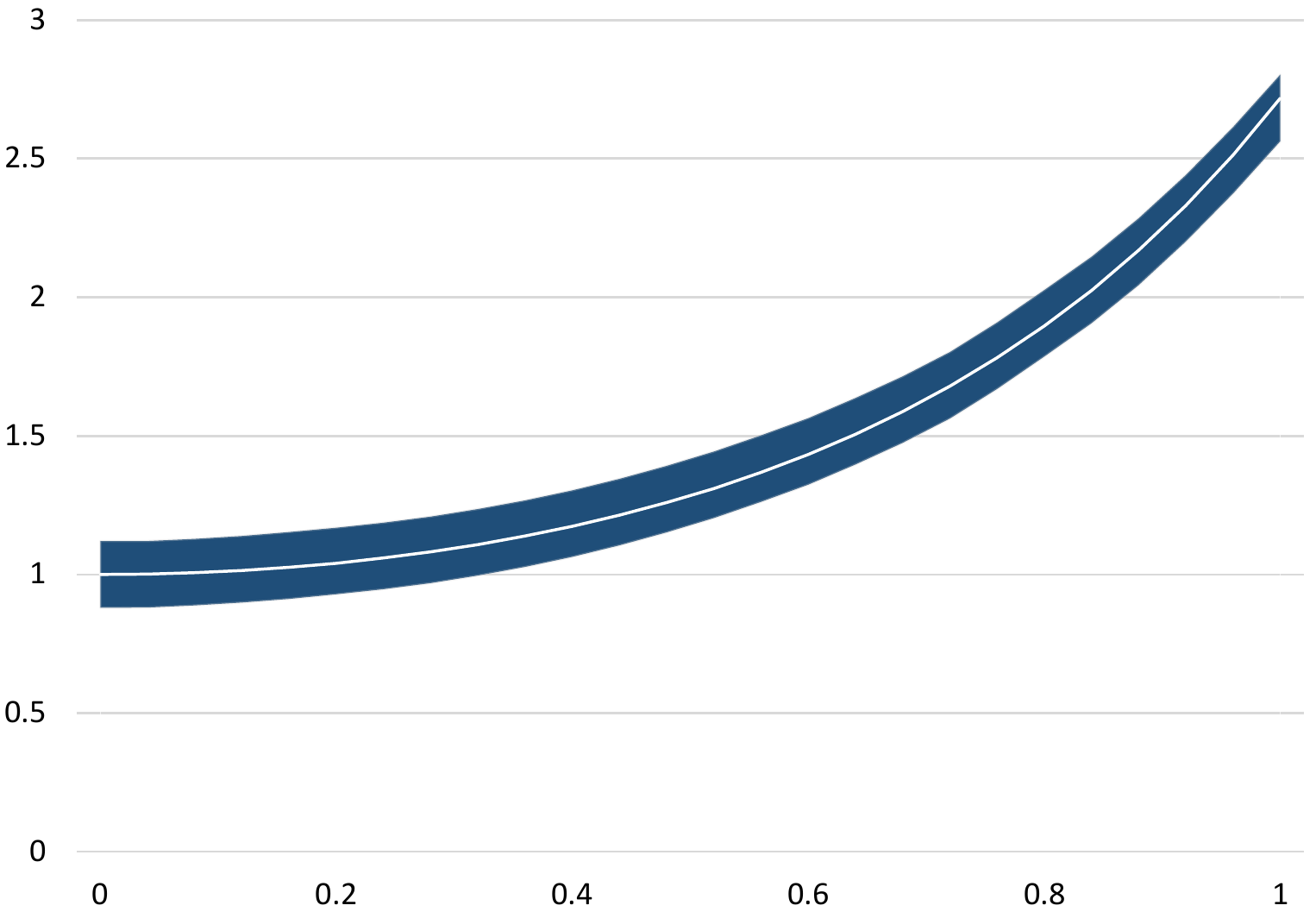}
        \subcaption{Confidence region based on the randomized explicit Euler method.}
    \end{subfigure}
    \begin{subfigure}{0.09\textwidth}
    \end{subfigure}
    \begin{subfigure}{0.45\textwidth}
    \centering
        \includegraphics[width=.95\linewidth]{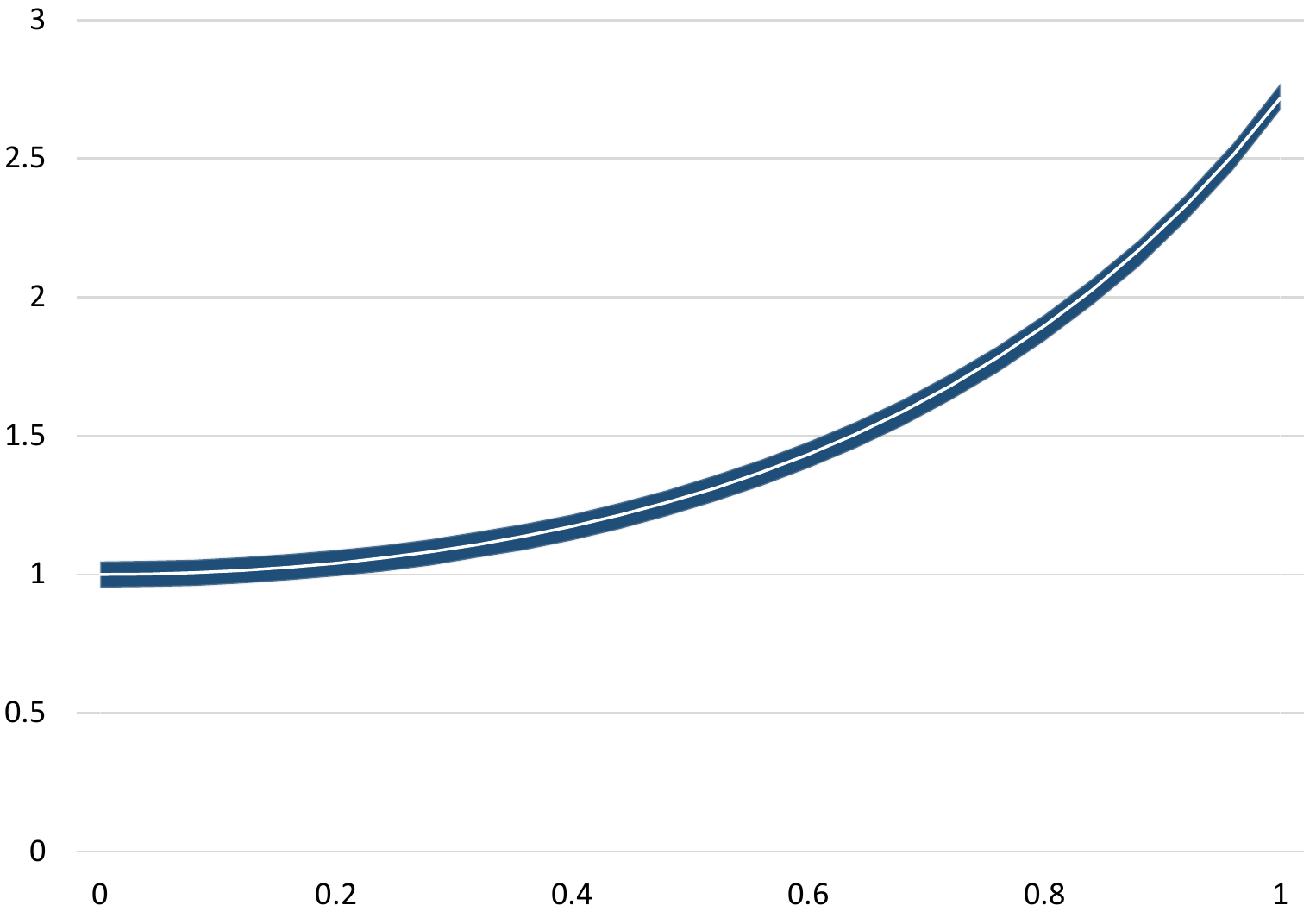}
        \subcaption{Confidence region based on the randomized two-stage Runge-Kutta method.}
    \end{subfigure}
    \caption{Confidence regions for the solution of the test problem \eqref{eq:A} with confidence level $1-\varepsilon=0.95$.}
    \label{fig:conf-reg-A}
\end{figure}

\begin{figure}[H]
    \centering
    \begin{subfigure}{0.45\textwidth}
    \centering
        \includegraphics[width=.95\linewidth]{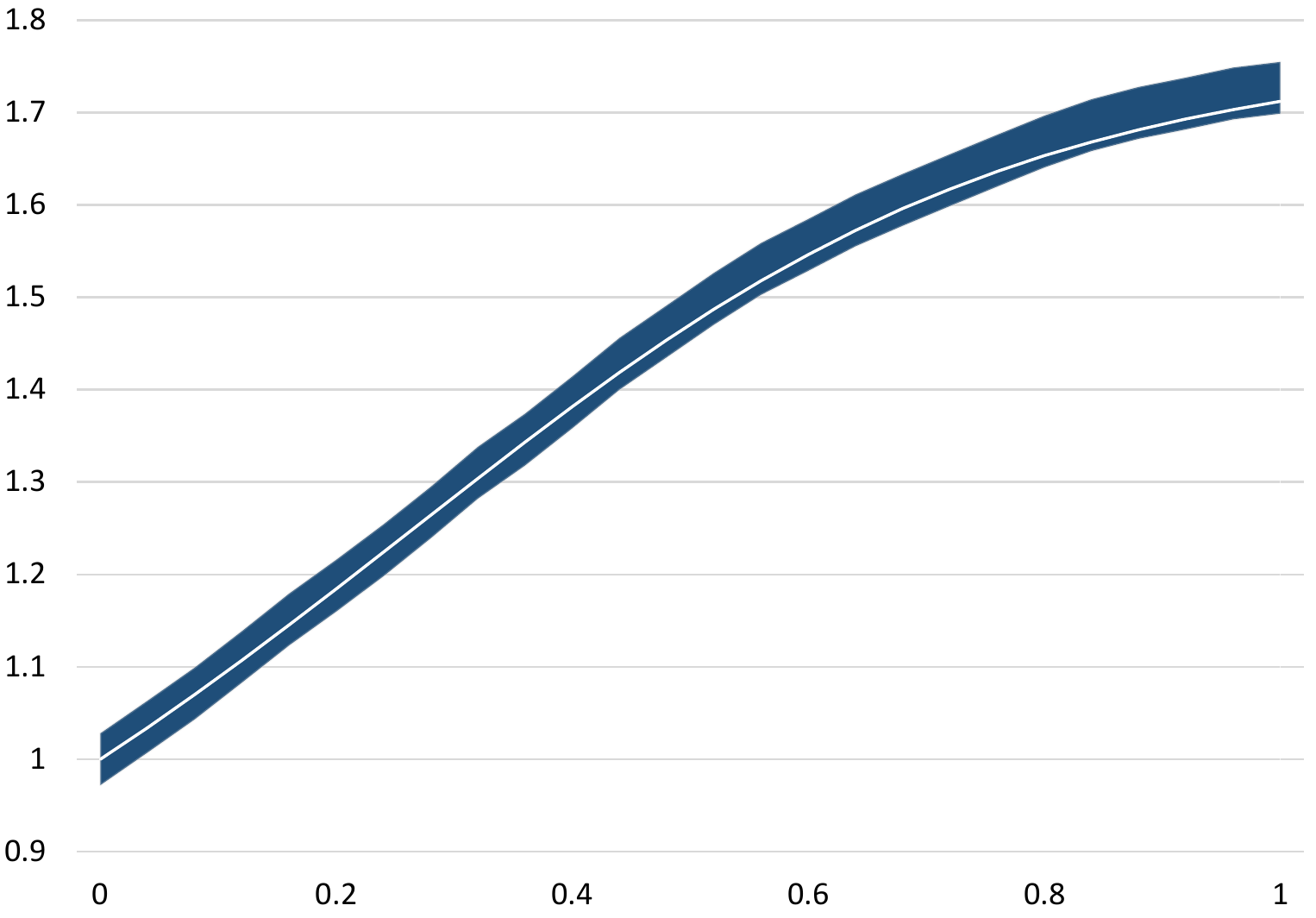}
        \subcaption{Confidence region based on the randomized explicit Euler method.}
    \end{subfigure}
    \begin{subfigure}{0.09\textwidth}
    \end{subfigure}
    \begin{subfigure}{0.45\textwidth}
    \centering
        \includegraphics[width=.95\linewidth]{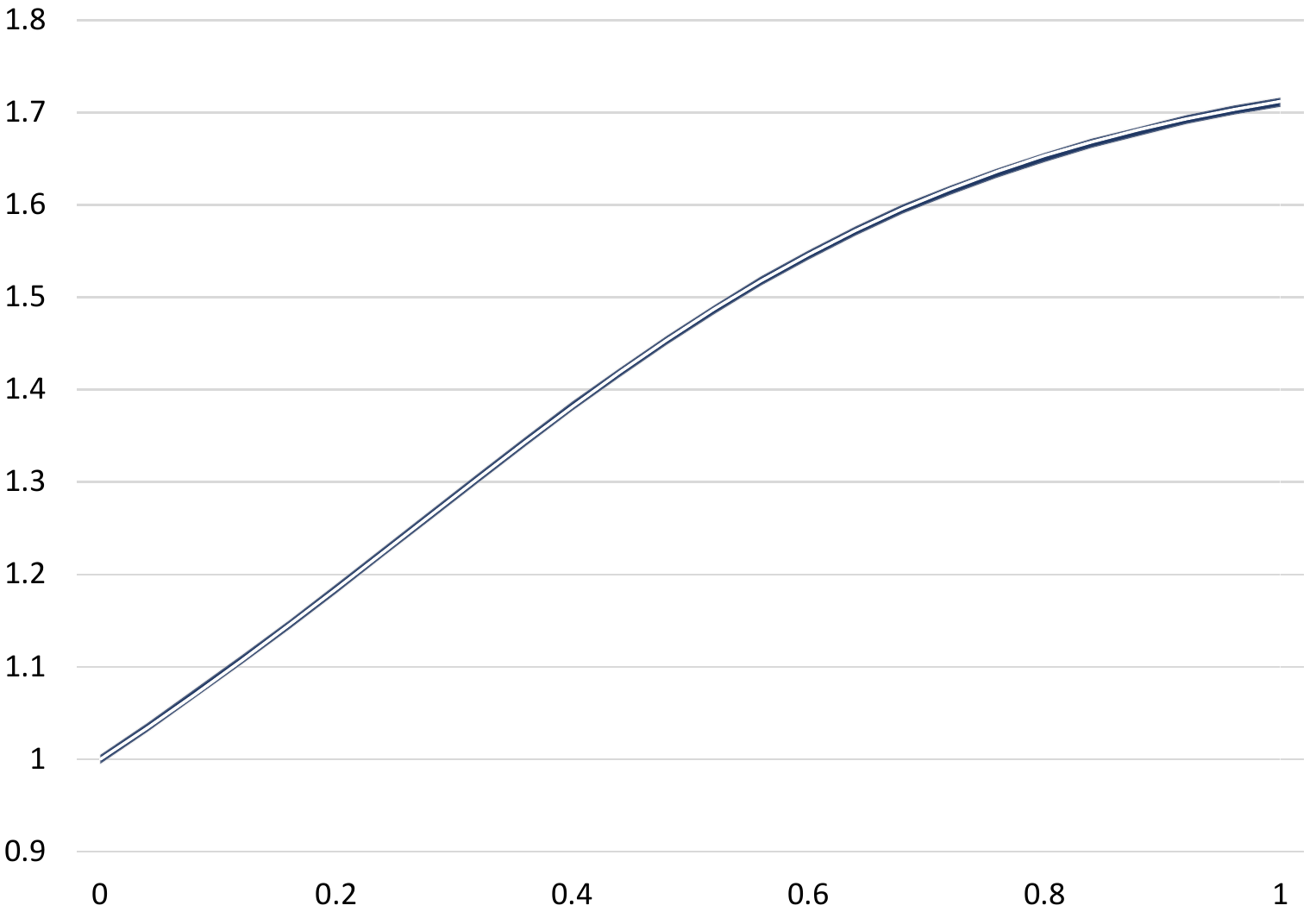}
        \subcaption{Confidence region based on the randomized two-stage Runge-Kutta method.}
    \end{subfigure}
    \caption{Confidence regions for the solution of the test problem \eqref{eq:B} with confidence level $1-\varepsilon=0.95$.}
    \label{fig:conf-reg-B}
\end{figure}

As we can see, for both test problems the randomized two-stage Runge-Kutta scheme generates more accurate confidence regions in comparison to the randomized explicit Euler scheme. This is due to the fact that $\gamma^{EE}=1<\frac32=\gamma^{RK}$.

\section{Conclusions and future work}

In this paper we have shown that the probability of the exceptional set for a class of randomized algorithms admitting a particular error decomposition has an exponential decay (see Theorem \ref{general_theorem}). A uniform almost sure convergence of such algorithms to the exact solution has been established when the step size and the noise parameter tend to $0$ (see Remark \ref{remark1}). Furthermore, we have used Theorem \ref{general_theorem} to design a confidence region for the exact solution of the IVP, uniform with respect to $\max\{h,\delta\}$ (see Corollary \ref{conf_reg} and Remark \ref{remark3}).

The general setting which has been considered comprises randomized explicit and implicit Euler schemes, and the randomized two-stage Runge-Kutta scheme under inexact information (see Corollaries \ref{cor:explicit}--\ref{cor:rk}). Theorem \ref{general_theorem} covers also the family of Taylor schemes under exact information, which has been investigated in \cite{HeinMilla}.

Our future plans include further research related to the probabilistic distribution of the error of randomized algorithms for ODEs, e.g. investigation of its asymptotic behaviour. 

\ 

\bf Acknowledgments. \rm The author would like to thank Professor Paweł Przybyłowicz for many inspiring discussions while preparing this manuscript.

This research was funded in whole or in part by the National Science Centre, Poland, under project 2021/41/N/ST1/00135.


\end{document}